\documentclass[11pt]{amsart}
\usepackage{amsmath}
\usepackage{amssymb}
\usepackage{amsthm}
\usepackage{graphicx}

\usepackage{color}
\long\def\blue#1{\textcolor {black}{#1}}

\newcommand{\C}{\mathbb{C}}
\newcommand{\Q}{\mathbb{Q}}

\newcommand{\Z}{\mathbb{Z}}

\newcommand{\F}{\mathbb{F}}

\newcommand{\Mat}{\mathrm{Mat}}
\newcommand{\Aut}{\mathrm{Aut}}
\newcommand{\End}{\mathrm{End}}

\newcommand{\Id}{\mathrm{Id}}
\newcommand{\EL}{\mathrm{EL}}
\newcommand{\SL}{\mathrm{SL}}

\newcommand{\GL}{\mathrm{GL}}
\newcommand{\GE}{\mathrm{GE}}

\newcommand{\St}{\mathrm{St}}

\newtheorem{lem}{Lemma}
\newtheorem{theorem}[lem]{Theorem}
\newtheorem{cor}[lem]{Corollary}
\newtheorem{rem}[lem]{Remark}
\newtheorem{defi}[lem]{Definition}

\newcommand{\la}{\langle}
\newcommand{\ra}{\rangle}

\begin{document}
\title{Nonlinearity of matrix groups}
\author{Martin Kassabov and Mark Sapir }
\thanks{The research of the
first author was supported in part by the NSF grant DMS
0600244 and by the Centennial
Fellowship from the American Mathematics Society.
The work of the second author
was  supported in part by the NSF grant DMS 0700811.}
\maketitle

The aim of this short note is to answer a question by Guoliang Yu of
whether the group $\EL_3(\Z\la x,y \ra)$, where $\Z\la
x,y\ra$ is the free (non-commutative) ring, has any faithful finite dimensional linear
representations over a field. Recall that for every (associative
unitary) ring $R$ the group $\EL_n(R)$ is \blue{the subgroup of $\GL_n(R)$}
generated by all $n\times n$-elementary matrices
$x_{ij}(r)=\Id + r e_{ij}$ ($r\in R$, $1\le i\ne j\le n$).
Clearly, if $R$ has a faithful finite dimensional linear representation
over a field, then the group $\EL_n(R)$ also has a faithful finite dimensional linear
representation over the same field. \blue{The conclusion is true even if
$R$ has an ideal of finite index that has a faithful
finite dimensional representation (see Theorem~\ref{th1}).}

The converse implication (which
would imply the negative answer to G.~Yu's question) should have
been known for many years, but we could not find it in the
literature. There are many results about isomorphisms
between various matrix groups over (mostly commutative\blue{)} rings from the
original results of Mal'cev~\cite{Mal} to results of O'\blue{Meara}~\cite{OM}
to Mostow rigidity results~\cite{Mos}.

There are also many result about homomorphisms of  one general
matrix group into another. Churkin~\cite{Chu} proved that the wreath
product $\Z\wr \Z^n$ embeds into a matrix group over a field $K$ of
characteristic $0$ if and only if the transcendence degree of $K$
over its prime subfield is at least $n$ \blue{(a similar result is proved
in the case of positive characteristic)}. Hence $\SL_n(K)$ cannot
embed into $\SL_m(K')$ if $K, K'$ are fields of characteristic $0$
and the transcendence degree of $K$ is bigger than the transcendence
degree of $K'$.
\blue{ Much stronger non-embeddability results for general
linear groups over fields follow from the main result of Borel-Tits~\cite{BT}
(we are grateful to Yves de Cornulier for pointing out to this reference).
See also surveys~\cite{JWW} and~\cite{HJW}.}

The main result of the note is the following:

\begin{theorem}
\label{th1}
(a) Let $R$ be an associative unitary ring, $k\ge 3$.
The group $\EL_k(R)$ has a faithful finite dimensional
representation over $\C$ if and only if $R$ has a finite index
ideal $I$ that admits a faithful finite dimensional representation
over $\C$.

(b) The group $\EL_3(\Z\la x,y \ra)$ does not have a faithful finite
dimensional representation over any field.
\end{theorem}

The proof of this theorem is given at the end of the paper (after
Remark~\ref{rem15}). Part (a) of Theorem~\ref{th1} does not hold if
we replace $\C$ by a field of positive characteristic (see
Remark~\ref{rem}). \blue{Note that the ``finite index" condition in
part (a) of Theorem~\ref{th1} is necessary because there are finite
rings (say, the endomorphism ring of the Abelian group $\Z/p\Z\times
\Z/p^2\Z$ where $p$ is a prime) that do not have any finite
dimensional faithful representations over fields and even over any
commutative rings~\cite{Berg}.} \blue{Also parts (a) and (b) of Theorem~\ref{th1}  do not
hold for $\EL_2$. Indeed, by Exercise 2 of Section 2.7 in \cite{Co},
the group $\GE_2(\Q\la x,y\ra)$ generated by the elementary and diagonal matrices of $\GL_2(\Q\la x,y\ra)$, is isomorphic to the group
$\GE_2(\Q[x])$ which is a subgroup of $\GL_2(\C)$, and so $\EL_2(\Z\la x,y\ra)<\GE_2(\Q\la x, y\ra)$ is linear.}

Let $\pi : \EL_k(R) \to \GL_n(K)$ be a linear representation of the
group $\EL_k(R)$, where $K$ is an algebraically closed field, \blue{$k\ge 3$}.

\begin{defi}
Let $U$ denote the set $U=\{\pi(x_{13}(r)) \mid r\in R\}$ and $V$ be
the Zariski closure of $U$. By construction $V$ is an algebraic
variety.
\end{defi}

\begin{theorem}
\label{th2}
There exist two distinguished elements $\mathbf 0$ and $\mathbf 1$ in $V$ and
polynomial maps $+, \times :  V \times V \to V$,  $-: V \to V$,
which give $V$ a structure of an associative ring.
Moreover the map
$\rho : R \to U\subset V$ defined by $\rho(r) = \pi(x_{13}(r))$ is
a ring homomorphism.
\end{theorem}
\begin{proof}
Define \blue{the ``addition''} $+ : U \times U \to U$ as follows $u_1 + u_2 :=  u_1u_2$,
where the multiplication on
the right is the one in the group $\GL_n(K)$. It is clear that this map is given by some algebraic
function therefore it extends to a polynomial map on $V \times V$.
Similarly we can define a map
$-: V \to V$ as the extension of the inversion $u \to u^{-1}$.
Notice that by construction we have the identities
$$
\rho(r_1) + \rho(r_2)  = \rho(r_1+ r_2)
\quad \mbox{and} \quad
-\rho(r)= \rho(- r),
$$
i.e., the map $\rho: R \to U$ is \blue{a} homomorphism between Abelian
groups. The identity element of $\GL_k$ is in $U\subset V$ and we
will denote it as the distinguished element $\mathbf{0} \in V$, since \blue{it} is
the identify element of $U$ with respect to the addition.

These two operations turn $V$ into an Abelian group\blue{:} since all the axioms
are satisfied on the Zariski dense set $U$  \blue{they} are
satisfied on the whole variety $V$.

\medskip

In order to define the \blue{``multiplication''} we need to use two special
elements $w_{23}$ and $w_{12}$ in $\EL_k(R)$ which have the
properties
$$
w_{12} x_{13}(r) w_{12}^{-1} = x_{23}(r)
\quad \mbox{and} \quad
w_{23} x_{13}(r) w_{23}^{-1} = x_{12}(r)
$$
The existence of these elements is well know\blue{n} and they can be easily
written as product of generators in $\EL_k(R)$, for example we can
take \blue{the} matrices (embedded in \blue{the} top left corner
if $\EL_k(R)$).
$$
w_{12} = \left(\begin{array}{ccc} 0& -1 & 0 \\ 1 & 0 & 0 \\ 0 & 0 & 1 \end{array} \right)
\quad \mbox{and} \quad
w_{23} = \left(\begin{array}{ccc} 1 & 0 & 0 \\ 0 & 0 & 1 \\ 0 & -1 & 0 \end{array} \right)
$$

Now we can define the algebraic map $\times : U \times U \to
\GL_n(K)$ as follows
$$
u_1 \times u_2 :=  [w_{23}u_1w_{23}^{-1}, w_{12}u_2w_{12}^{-1}]
$$
The commutator relation $[x_{12}(r),x_{23}(s)] = x_{13}(rs)$ implies that
$$
\rho(r_1) \times \rho(r_2)  = \rho(r_1.r_2),
$$
thus $\times$ is a map from $U \times U$ to $U$ and can be extended to a
polynomial map from $V \times V$ to $V$.
The element $\rho(1)$ plays the role of the unit with respect to
this multiplication and we will call it $\mathbf{1}\in V$\blue{.}

The same argument as before shows that $\mathbf{0},\mathbf{1}$
and the maps $+, -$ and $\times$ turn $V$ into an associative ring with a unit.
\end{proof}

\begin{lem}
\label{lem}
Let $V$ be an algebraic variety with two algebraic operations,
given by polynomial  functions,
which turn it into an associative ring
with $1$. 
Then:

\blue{
(a) any point on $V$ is non-singular,
thus the irreducible components of $V$ do not intersect;
}

\noindent Let $V_0$ denote the irreducible \blue{(}connected) component of $\mathbf{0}$ in $V$ then:

(b)  $V_0$ is a two-sided ideal in $V$;

(c) the quotient $V/V_0$ is a finite ring.
\end{lem}
\begin{proof}
\blue{(a)}
The structure of an Abelian group on $V$ with respect to the addition implies
that the automorphism group of the variety $V$ acts transitively on the points,
therefore all points are non-singular;

(b)
For any $v\in V$ the closure of $v\times V_0$ is a irreducible sub-variety $V$
(since it is an image of a irreducible one) which contains $\mathbf{0}$,
therefore it is a subset of $V_0$.
This shows that $V_0$ is a left ideal in $V$.
Similar argument shows that $V_0$ is a right ideal;

(c)
It is a classical result that any algebraic variety
has only finitely many irreducible components.
\end{proof}

\begin{lem}
\label{linear}
Let $V$ be an algebraic variety over $\C$, with two
algebraic operations which turn it into an associative ring with
$1$. Then the irreducible component  $V_0$ of $V$ is isomorphic to a
finite dimensional algebra over $\C$, i.e., the ring $V$ is
virtually linear over $\C$.
\end{lem}
\begin{proof}
Note that the additive group
$V_+$ of $V$ is an Abelian Lie group
over $\C$.\footnote{We consider the topology on $V$ induced by the usual topology on \blue{$\C^n$},
instead of the Zariski topology. This is one of the reason\blue{s} why this argument does
not work over fields of positive characteristic.} By~\cite{Pontr} $V_0$ is a product of a finite number of
copies of $\C$ and a finite number of 1-dimensional 
\blue{tori}.
Therefore the fundamental group $\Gamma$ of $V_0$ (based at $\mathbf 0$) is
isomorphic to $\Z^k$ for some
$k<\infty$, and the product of
any two loops in $\Gamma$ is the same as their point-wise sum in
$V_0$.

Multiplication by an element in $V$ induces an endomorphism of
$\Gamma$ and so we have a map  $\phi$ from  $V$ to the
endomorphism ring $\End(\Gamma)$ of $\Gamma$. This map is continuous and a ring
homomorphism because it preserves multiplication by construction and
the distributive law implies that $\phi$ send the sum of the loops
to the point-wise sum of their images. The endomorphism ring is
discrete, therefore the image of $V_0$ is trivial and $\phi$ factors
through a map $\bar \phi : V/V_0 \to \End(\Gamma)$. The ring
$\End(\Gamma)$ does not have any finite sub-rings since the
characteristic is $0$, unless $\Gamma$ is the trivial group,
because the order of the identity \blue{in $\End(\Gamma)\simeq \Mat_k(\Z)$}
is infinite.
\blue{Thus $\Gamma$ is trivial and 
$V_0$ is a simply connected Abelian Lie group over $\C$.}
Therefore $V_0$ is isomorphic to a finite dimensional vector space over $\C$.
The distributive laws imply that  multiplication on $V_0$ is
bilinear, i.e., $V_0$ is a finite dimensional algebra over $\C$.
\end{proof}

\begin{rem}
\label{strange_ring}
\rm{ The analog of Lemma~\ref{linear} is not
true in the case of positive characteristic. It is possible to
construct examples where the exponent of the additive group of $V$
is finite but is not equal to the characteristic of the field.

Here is one simple example \blue{(it is somewhat similar to the example from~\cite{Berg} mentioned above)}: Let $K$ be an infinite field of
characteristic $2$ and let $V = \blue{K\times K}$ with the following operations:
$$
(a,b) + (c,d) = (a + c, ac + b +d)
\quad\quad
(a,b) \times (c,d) = (ac, bc^2 + a^2 d)
$$
\blue{One can verify directly that $V$ is a commutative ring.}
The elements $(0,b)$ form an ideal
$I$ with zero multiplication, $V/I$ is isomorphic as a ring to the
field $K$ (identified as a set with $\{(a,0)\mid a\in K\}$), the
action of $V/I$ on $I$ is given by $(a,0)(0,d)=(0,a^2d)$. Every
element of the form $(a,b)$, $a\ne 0$, is invertible (the inverse is
$(a^{-1}, \frac{b}{a^2})$), \blue{i.e., $V$ is a local ring with a maximal ideal} $I$.
Therefore that ring does not have proper
ideals of finite index. This ring is not linear over any field since
all elements of the form $(a,b)$, $a\ne 0$, have ``additive'' order
$4$. Hence $V$ is not virtually linear.
}
\end{rem}

\begin{cor}
\label{nonemebdding_char0}
Let $V$ be an algebraic variety over a
field of characteristic $0$, with two algebraic operations which
turn it into an associative ring with $1$. Then any ring
homomorphism $\phi : \Z\la x,y\ra \to V$ has non-trivial kernel.
\end{cor}
\begin{proof}
By the previous lemma $V$ is virtually linear therefore it satisfies
some polynomial identity~\cite{Rowen}, but the ring $\Z\la x,y\ra$
does not satisfy any polynomial identity~\cite{Rowen}. Therefore
$\phi$ is not injective.
\end{proof}


\blue{ Using Lemma~\ref{linear} we can easily recover a significant
part of the result by Chen~\cite{Ch}. Let $D$ be a (noncommutative)
division ring, and consider a group homomorphism from ${\rm
SL}_n(D)$ to $G(k)$ with Zariski-dense image, where $G$ is a simple
algebraic group defined over a field $k$. Then $D$ must be
finite-dimensional over its center. The following corollary recovers
that result in the case of characteristic 0 and $n\ge 3$.} \blue{
\begin{cor}
Let $D$ be a non-commutative division ring with characteristic $0$.
Then the group $\EL_n(D)$, $n\ge 3$, has nontrivial finite
dimensional representations in characteristic $0$ if and only if $D$
is finite dimensional over its center.
\end{cor}
\begin{proof}
Let $\pi: \EL_n(D) \to \GL_n(k)$ be a nontrivial representation of
the group $\EL_n(D)$. Therefore the map $\rho: D \to U$ is a ring
homomorphism with a non-trivial image and $\rho$ has to be an
isomorphism since $D$ does not have any non-trivial ideals. By
Lemma~\ref{linear} $D$ is linear and thus it satisfies some
polynomial identity. Finally by a theorem of Kaplansky~\cite{He}
every division algebra $D$ satisfying a polynomial identity is
finite dimensional over its center. The other direction is obvious.
\end{proof}
}

\begin{lem}
\label{l7}
Let $V$ be an algebraic variety over a field $K$ (of arbitrary
characteristic) with two algebraic operations which turn it into an
associative ring with $1$. If $V$ is irreducible then the
multiplicative group of $V$ is linear over $K$.
\end{lem}
\begin{proof}
Let $A$ denote the ring of germs of rational functions on $V_0$
defined around the point $\mathbf 0$. Let $I$ be the maximal ideal in $A$
consisting of germs that are $0$ at $\mathbf 0$.
By Lemma~\ref{lem}, all points of $V$, including the point $\mathbf 0$, are
non-singular. Therefore $I/I^2$ is a finite dimensional vector space
over the field $A/I=K$, and the dimension coincides with the
dimension of $V$.

The left multiplication $l_v$ by any $v\in V$ defines an algebraic
map $V_0 \to V_0$ which fixes $\mathbf 0$ therefore it \blue{induces an} 
endomorphism $l_v: A \to A$. It is clear that these maps define a
group homomorphism $\psi: V^* \to \Aut(A)$ by \blue{$(\psi(v)(f))(x) = f(l_v(x))$},
where $V^*$ is the set all invertible elements in $V$. The kernel $S$ of
the map $\psi$ consists of all elements $v$ in $V^*$ such that
$(v - 1) \times V_0 =\mathbf 0$, because the triviality of $l_v$ implies that the
multiplication by $v$ gives the identity map from $V_0$ to $V_0$. If
$V$ is connected then $V_0$ contains $1$ thus the only element in
the kernel of $\psi$ is the identity.

Consider the maps  $\psi_n: V^* \to \Aut(A/I^n)$ induced by $\psi$
and their kernels $S_n=\ker \psi_n$. By construction \blue{the sets} 
$S_n$ form a decreasing sequence of sub-varieties of $V$ and that
$\cap_n S_n = S$. By the Noetherian property, we have that there
exist $M > 0$ such that $S=S_M$, i.e., the map $\psi_M$ is
injective.

The group $\Aut(A/I^M)$ is linear over $K$ because it is inside the
group of all linear transformations of $A/I^M$ considered as a
(finite dimensional) vector space over $K$,
\blue{i.e., $\psi_M$ is a faithful
linear representation of $V$.}
\end{proof}

\begin{rem}
\label{rem}
\rm{Let $V$ be the variety with the ring structure constructed in
Remark~\ref{strange_ring}. \blue{The group $\EL_3(V)$ is a
subgroup of the multiplicative semigroup of the ring of $3\times 3$ matrices over $V$,
which is an algebraic variety (isomorphic to $K^{18}$, where the addition and
the multiplication are given by some polynomial functions of degree
$4$).
Lemma~\ref{l7} implies that $\EL_3(V)$ is linear group over $K$.}
Thus, there exist\blue{s} a ring $R$ which is not (virtually) linear
over any field, but the group $\EL_3(R)$ is linear.
Hence part (a) of Theorem~\ref{th1}
does not hold in the case of positive characteristic.}
\end{rem}

\blue{\medskip}

The result in \blue{C}orollary~\ref{nonemebdding_char0} also holds in the
case of positive characteristic, but the argument is different.

\begin{theorem}
\label{nonemebdding_any_char}\label{th9} \label{inclusion}
Let $V$
be an algebraic variety with two algebraic operations which turn it
into an associative ring. Then any ring homomorphism
$\phi : \Z\la x,y\ra \to V$ has \blue{a} non-trivial kernel.
\end{theorem}
\begin{proof}
Let $k$ be the dimension of $V$ and let assume that the map $\phi$ is injective.
Let $s_l$ denote the symmetric function on $l$ arguments, i.e.,
$$
s_l(x_1,\dots,x_l) = \sum_{\sigma \in S_n} (-1)^\sigma \prod x_{\sigma(i)}
$$

Pick elements $r_1,r_2,\dots,r_{k+1}$ such that $s_l(r_1,\dots,r_l)
$ is not $0$
in the ring $R=\Z\la x,y\ra$ for any $l \leq k+1$
(for example we can take $r_i = x y^{i+1}$). Let $M_l$ denote
the $\Z$ span of the elements $r_1,\dots, r_l$ in the ring $R$ and
let $N_l$ be the Zariski closure of $\phi(M_l)$ in $V$.
\begin{lem}
The symmetric function $s_{l+1}$ is zero when evaluated on any $l+1$
elements in $M_l$.
\end{lem}
\begin{proof}
The polynomial $s_{l+1}$ is linear in every variable and
anti-symmetric, and $M_l$ is spanned by less that $l+1$ elements.
\end{proof}
This immediately implies\blue{:}
\begin{cor}
The symmetric function $s_{l+1}$ is zero when evaluated on any $l+1$ elements in $N_l$.
\end{cor}
\begin{lem}
For any $l$ we have that $\dim N_{l} > \dim N_{l-1}$.
\end{lem}
\begin{proof}
Let $N_{l,i}$ denote the set $i. \phi(r_l) + N_{l-1}$ for a positive
integer $i$ (here $i.r$ denotes the sum $r+r+ \dots +r$). Using the
fact that the operation $+$ is an algebraic function, we can
conclude that this is a sub\blue{-}variety of $N_{l}$ and
$\dim N_{l,i} = \dim N_{l-1}$ because the algebraic map $v \to i. \phi(r_l)  + v$ is
a bijection from $V$ to $V$. Let us show that these subvarieties are
disjoint: assume that $i_1.\phi(r_l) + v_1 = i_2.\phi(r_l) + v_2$
for some different integers $i_1$ and $i_2$ and some points $v_1,v_2
\in N_{l-1}$. Using the linearity of the symmetric function $s_{l}$
we have
$$
(i_2-i_1).s_{l}\left(\phi(r_1),\dots,\phi(r_{l-1}), \phi(r_l)\right)=
s_{l}\left(\phi(r_1),\dots,\phi(r_{l-1}), v_1 -v_2\right)=
$$
$$
=
s_{l}\left(\phi(r_1),\dots,\phi(r_{l-1}), v_1\right) -
s_{l}\left(\phi(r_1),\dots,\phi(r_{l-1}), v_2\right)=
\mathbf{0}
$$
because $s_l$ is trivial on $N_{l-1}$.
However this contradicts the choice of the elements $r_i$ and the injectivity of $\phi$ because
$$
(i_2-i_1).s_{l}\left(\phi(r_1),\dots,\phi(r_{l-1}), \phi(r_l)\right)=
\phi\left( (i_2-i_1).s_{l}(r_1,\dots,t_l) \right) \not =
\mathbf{0}.
$$
Thus $N_l$ contains infinitely many subvarieties of dimension equal to the one of $N_{l+1}$,
which is only possible if  $\dim N_{l} > \dim N_{l-1}$.
\end{proof}
The above lemma yields:
\begin{cor}
The dimension of $N_l$ is greater than or equal to $l$.
\end{cor}

This is a contradiction because by construction $N_{k+1} \subset V$
and $\dim V =k < k+1 \leq\dim N_{k+1}$, which completes the proof of
Theorem~\ref{inclusion}.
\end{proof}

\begin{rem} \label{rem15}
\rm{
It is not clear if it is possible to embed $\F_p\la x,y \ra$ into an
algebraic variety with a ring structure over a \blue{field} of positive characteristic.
(The above argument only works if the ``base ring'' contains
$\Z$.)
}
\end{rem}

\blue{
Before completing the proof of Theorem~\ref{th1}, we need to prove
a lemma about the Steinberg groups $\St_n(R)$ where $R$ is an associative unitary ring.
This group has (formal) generators $x_{ij}(r)$ for $1\leq i\not=j
\leq n$ and $r\in R$, which satisfy the following commutator
relations:
$$
\begin{array}{ll}
x_{ij}(r)x_{ij}(s) = x_{ij}(r+s)  & {} \\
{}[x_{ij}(r),x_{pq}(s)] = 1 & \mbox{ if } i\not=q, j\not=p \\
{}[x_{ij}(r),x_{jk}(s)] =x_{ik}(s) & \mbox{ if }  i \not=k \\
\end{array}
$$
There is a surjection from $\St_n(R)$ onto $\EL_n(R)$ mapping
$x_{ij}(r)$ to $\Id + r e_{ij}$. The kernel of this surjection is
denoted by $K_{2,n}(R)$.
}

\blue{The following lemma is fairly standard and probably well known.
We are grateful to Nikolai Vavilov for correcting the original proof of that lemma.}
\blue{
\begin{lem}\label{St}
If $R$ is a finite ring, then the Steinberg group $\St_n(R)$ is finite for any $n\ge 3$.
\end{lem}
}
\blue{
\begin{proof}
Using results of Vaserstein~\cite{Va} and Milnor~\cite{Mi}
we deduce that
in the case of a finite ring $R$,
the groups $K_{2,n}(R)$ do not depend on $n\geq 3$ and are finite and central in $\St_n(R)$
(because the stable range of a finite ring is $1$).
Therefore, the Steinberg group
$\St_n(R)$, $n\ge 3$, is a \blue{perfect} extension of a finite group $K_{2,n}(R)$ by the finite group $\EL_n(R)$ and 
thus
is also finite.
\end{proof}
}

\begin{proof}[Proof of Theorem~\ref{th1}.]
(a) Suppose that $G=\EL_3(R)$ is
linear over a field $\C$. Then by Theorem~\ref{th2} $R$ embeds into
a ring that is a variety over $\C$. By Lemma~\ref{linear}, then $R$
has a finite index ideal that is linear over $\C$.

Suppose now that $R$ has a finite index ideal $I$ that is linear
over $\C$. Consider the {\em congruence subgroup} $G_I$ of $G$
corresponding to $I$, that is the subgroup generated by all
$x_{ij}(r)$, $r\in I$.
\blue{The subgroup $G_I$ has a finite index
in $G$, because the quotient $G/G_I$ is a homomorphic image of the
Steinberg group $\St_3(R/I)$ which is finite by Lemma~\ref{St}.}
Also, $G_I$ is linear over $\C$, therefore $G$ is linear over $\C$
(consider the representation induced by the faithful representation
of $G_I$).

(b) Suppose that $G=\EL_{\blue{3}}(\Z\la x,y \ra )$ is linear over any field $K$.
Again by Theorem~\ref{th2}, then $\Z\la x,y\ra $ embeds into a ring that
is a finite dimensional algebraic variety over $K$. By Theorem~\ref{th9},
that is impossible, a contradiction.
\end{proof}

The following theorem can be proved in the same manner as
Theorem~\ref{th1}.

\begin{theorem}\label{st}
(a) If the group $\St_3(R)$ is linear over
$\C$, then $R$ has a finite index ideal that is linear over $\C$.

(b) The group $\St_3(\Z\la x,y \ra)$ is not linear over any field $K$.
\end{theorem}

\bigskip

\begin{minipage}[t]{2.5  in}
Martin Kassabov\\
Department of Mathematics, \\ Cornell University, \\
Ithaca, NY 14853-4201 USA\\
kassabov@math.cornell.edu
\end{minipage}
\begin{minipage}[t]{3 in}
\noindent Mark V. Sapir\\ Department of Mathematics\\
Vanderbilt University\\
Nashville, TN 37240, USA\\
m.sapir@vanderbilt.edu
\end{minipage}
\end{document}